\documentclass[11pt]{article}
\usepackage{amsmath, amsthm, amssymb, eucal, setspace, mathrsfs, extarrows}

\makeatletter
\renewcommand\section{\@startsection{section}{1}{\z@}%
 						{-3.5ex \@plus -1ex \@minus -.2ex}
						{2ex \@plus.2ex}
						{\large\bfseries}}
\renewcommand\subsection{\@ifstar
						{\setcounter{subsection}{\value{equation}}
					\@startsection{subsection}{2}{\z@}
                          {1.75ex \@plus.5ex \@minus.2ex}%
                           {-.4em}		
					\textit*}
					{\setcounter{subsection}{\value{equation}}
						\stepcounter{equation}
					\@startsection{subsection}{2}{\z@}
                          {1.75ex \@plus.5ex \@minus.2ex}%
                           {-.4em}		
					\textit}}
\def\@seccntformat#1{\@ifundefined{#1@cntformat}%
	{\csname the#1\endcsname\quad} 
	{\csname #1@cntformat\endcsname}} 
\def\section@cntformat{\thesection.~} 
\def\subsection@cntformat{(\thesubsection)\ }
\renewcommand*\l@section{\mdseries\small\@dottedtocline{1}{1.5em}{2em}}
\makeatother

\numberwithin{equation}{section}
\theoremstyle{plain}
\newtheorem{maintheorem}{Theorem}
\swapnumbers

\newtheorem{corollary}[equation]{Corollary}
\newtheorem{lemma}[equation]{Lemma}
\newtheorem{proposition}[equation]{Proposition}
\theoremstyle{definition}

\theoremstyle{remark}
\newtheorem{remark}[equation]{Remark}

\newcommand{\cA}{\mathscr{A}}

\newcommand{\cC}{\mathscr{C}}

\newcommand{\cI}{\mathscr{I}}

\newcommand{\cN}{\mathscr{N}}
\newcommand{\cO}{\mathscr{O}}

\newcommand{\frg}{\mathfrak{g}}
\newcommand{\frh}{\mathfrak{h}}

\newcommand{\bC}{\mathbb{C}}
\newcommand{\bP}{\mathbb{P}}

\newcommand{\bR}{\mathbb{R}}

\newcommand{\bZ}{\mathbb{Z}}
\newcommand{\vep}{\varepsilon}
\newcommand{\ddisk}{-\!\!:\!\!-}

\title{\Large\textbf{The r\^{o}le of Coulomb branches in $2D$ gauge theory}}  										
\author{Constantin Teleman\footnote{UC Department of Mathematics, 970 Evans Hall, 
	Berkeley, CA 94720. \texttt{teleman@berkeley.edu}}}
\begin{document}
\maketitle
\begin{abstract}
\noindent
I give a  construction of the \emph{Coulomb branches} $\cC_{3,4}(G;E)$ of gauge theory 
in $3$ and~$4$ dimensions, defined by Nakajima \emph{et al.} \cite{n, bfn} for a compact Lie 
group $G$ and a polarisable quaternionic representation $E$. The manifolds $\cC(G;\mathbf{0})$ 
are abelian group schemes over the bases of regular adjoint $G_\bC$-orbits, respectively conjugacy 
classes, and $\cC(G;E)$ is glued together over the base from two copies of $\cC(G;\mathbf{0})$ shifted 
by a rational Lagrangian section $\vep_V$, representing the Euler class of the 
\emph{index bundle} of a polarisation 
$V$ of $E$. 
Extending the interpretation of $\cC_3(G;\mathbf{0})$ as  ``classifying space'' for topological $2$D 
gauge theories, I characterise  functions on $\cC_3(G;E)$ as operators on the equivariant 
quantum cohomologies of $M\times V$, for compact symplectic $G$-manifolds $M$. The non-commutative 
version has a similar description in terms of the $\Gamma$-class of $V$.
\end{abstract}

\section{Introduction}

Associated to a compact connected Lie group $G$ and a quaternionic representation $E$, 
there are expected to be \emph{Coulomb branches} $\cC_{3,4}(G;E)$ of $N=4$ SUSY gauge theory in 
dimensions $3$ and $4$, with matter fields in the representation $E$. They ought to be components of 
the moduli space of vacua, representing solutions of the monopole equations with 
singularities. Following early physics leads \cite{seibwit, ch} and more recent calculations \cite{chmz}, 
a precise definition for these spaces was proposed in the series of papers \cite{n, bfn} by Nakajima 
and collaborators in the case when $E$ is \emph{polarisable} (isomorphic to $V\oplus V^\vee$ for 
some complex representation $V$). Abelian groups were handled independently by Bullimore, Dimofte 
and Gaitto \cite{bdg} from a physics perspective, while the case of the zero representation had 
been developed in \cite{bfm}, although only later 
recognised as such \cite{taust, ticm}.  

The $\cC_{3,4}$ are expected to be hyperk\"ahler (insofar as this makes sense for singular spaces), 
with $\cC_3$ carrying an $\mathrm{SU}(2)$ hyperk\"ahler rotation. They are constructed in \cite{bfn} 
as algebraic Poisson spaces, with $\bC^\times$-action in the case of $\cC_3$. We shall rediscover 
them as such in a simpler construction, which illuminates their relevance to $2$-dimensional 
gauge theory: the $\cC_{3,4}$ for polarised  $E$ are built from their more basic versions 
for the zero representation $E=\mathbf{0}$. Specifically, they are affinisations of a 
space constructed by partial identification of two copies of $\cC(G;\mathbf{0})$. 
The identification is implemented by a Lagrangian shift along the fibers of the (Toda) 
integrable system structure of the $\cC_{3,4}$, and its effect is to impose growth conditions, 
selecting a subring of regular functions. The non-commutative versions quantise this 
Lagrangian shift of the $\cC_3$ into conjugation by the $\Gamma$-class of the representation
(respectively, a specialisation of its Jackson-$\Gamma_p$ version for $\cC_4$). 

 The reconstruction results, Theorems~\ref{c3}, \ref{c4} 
and \ref{cnc}, are more elementary than their $2$D gauge theory interpretation, but it is the latter 
which seems to give them meaning. In compromise, I have attempted to isolate the gauge theory comments 
(for which a rigorous treatment has not yet been published) into paragraphs whose omission does not harm 
the remaining mathematics. I have also separated the non-commutative version of the story  
into the final section: its meshing with quantum cohomology theory is still incomplete.

A pedestrian angle on this paper's results is the \emph{abelianisation} underlying the 
calculations --- a reduction to the Cartan subgroup $H$ and Weyl group $W$. This is seen 
in the description of the Euler Lagrangians \eqref{eulerdef} which are used to build the 
``material'' Coulomb branches from $\cC(G;\mathbf{0})$, and is closely related to the 
abelianised index formula in \cite{tw}, 
which ends up governing the Gauged Linear Sigma model (GLSM). Oversimplifying a bit, the 
interesting difference between $G$ and its abelian reduction is already contained in 
$\cC(G;\mathbf{0})$, the effect of adding a polarised representation being captured 
by a calculus reminiscent of toric geometry. Abelianization also has an explicit manifestation, 
similar to the Weyl character formula, in an isomorphism 
\[
\cC_{3,4}(G;E)\cong \cC_{3,4}\left(H;E\ominus (\frg/\frh)^{\oplus 2})\right)/W, 
\]
whenever the formal difference on the right is a genuine representation of $H$; a quick 
argument has been included in the appendix, as it appeared not to be well known.

A qualification is in order: the simple characterization applies to the variants of $\cC_{3,4}$ 
enhanced by the (complex) \emph{mass parameters} \cite{bdg}, or by the more general \emph{flavor 
symmetries} \cite[3(vii)]{bfn}. The original spaces are subsequently recovered by setting 
the mass parameters to zero; however, at least one parameter, effecting a compactification of 
$V$, must be initially turned on. The moral explanation is easily expressed in physics language, 
and in a way that can be 
made mathematically precise. What my construction does is characterize the $3$-dimensional 
topological gauge theories underlying the $\cC_{3,4}$ by means of their $2$D \emph{topological 
boundary theories} --- a characterization accurate enough, at least, to determine their expected 
Coulomb branches. For pure gauge theory ($E=\mathbf{0}$), I explained in \cite[\S6]{ticm} in 
what sense the ($A$-models of) flag varieties of $G$ supply a complete family of boundary 
theories, the Coulomb branch being akin to a direct integral of those: more precisely, 
it has a Lagrangian foliation by the mirrors of flag varieties. With $E$-matter added, 
a new boundary theory, the GLSM of $V$ by $G$ (again in the $A$-version) 
must be introduced as a factor, carrying the action of the matter fields. Since $V$ is not 
compact, this model must be regularized by the inclusion of mass parameters. There 
is a mathematically sound version of this statement: the GLSM is a $2$D TQFT 
over the ring of rational functions in the complex mass parameters, and  has 
singularities at zero mass.   

The same perspective points to a difficulty in extending these constructions when $E$ cannot 
be polarised. There is no \emph{a priori} reason why a $3$D TQFT should be 
characterized by its topological boundary theories; Chern-Simons theory (for general 
levels) is a notorious counter-example \cite{ks}. A $G$-invariant Lagrangian $V\subset E$ 
seems to provide (in addition to the flag varieties) a generating boundary condition for the 
$3$D gauge theory with matter --- specifically, it is a domain wall between $G$-gauge theory 
with and without matter. No substitute is apparent in general. 
Clearly this deserves further thought. One obstacle is that $3$D gauge theory 
gives only a partially defined TQFT, so its mathematical structure is incompletely 
settled, and the list of desiderata for a presumptive reconstruction is not known with clarity.

\begin{quote}\small
\emph{Acknowledgements.} I learned much about Coulomb branches from A.~Braverman, 
T.~Dimofte and H.~Nakajima, some of it during a stay at the Aspen Physics Center. 
The work was partially supported by NSF grants PHY-1066293 (Aspen), DMS--1406056, and 
by All Souls College, Oxford; an early version was first presented at the 2016 Clay 
workshop in Oxford.
\end{quote}

\section{Overview and key examples}
This section reviews the basic ingredients of the story and indicates the construction 
of Coulomb branches using $\mathrm{U}_1$ as an example. The full statements require 
more preparation, and are found in \S\ref{stat}.

\subsection{Background.} 
The complex-algebraic symplectic manifold $\cC_3(G;\mathbf{0})$ was introduced for general $G$ 
in \cite{bfm}; for $G=\mathrm{SU}_n$, it had been studied in \cite{athit}, in the guise of the 
moduli space of 
$\mathrm{SU}_2$ monopoles of charge $n$. The description most relevant for us is $\mathrm{Spec}\, 
H_*^G(\Omega G;\bC)$, the conjugation-equivariant homology of the based loop group $\Omega G$, 
with its Pontryagin product. From here, its r\^ole as a \emph{classifying space} for topological 
$2$-dimensional gauge theories was developed in \cite{taust, ticm}, where the space was denoted 
$\mathrm{BFM}(G^\vee)$. As we now recall, this virtue of $\cC_3(G;\mathbf{0})$ must be read 
in the sense of semiclassical symplectic calculus, and not as a spectral theorem \`a la 
Gelfand-Naimark. It gives the  ``mirror description'' of the gauged $A$-models in $2$ dimensions. 

\subsection{Relation with quantum cohomology.}
A partial summary of the classifying property of $\cC_3(G;\mathbf{0})$ is that its regular 
functions (sometimes called the \emph{ring of chiral operators}) act on the equivariant quantum 
cohomologies $QH^*_G(M)$ of compact $G$-Hamiltonian symplectic manifolds $M$, in a manner 
making the $E_2$ structure\footnote{Understood in the derived sense.} on $QH^*_G(M)$ compatible with 
the $E_3$ structure defined by the Poisson tensor on $\cC_3$. This lays out $QH^*_G(M)$ as a 
sheaf over $\cC_3$, which turns out to have Lagrangian support (Remark~\ref{lag} below). This 
construction generalises Seidel's theorem \cite{seid} on the action of $\pi_1G$ on $QH^*(M)$, 
as well as the shift operators on $QH^*$ and their equivariant extensions \cite{op}. In fact, 
these latter ingredients are the ``leading order'' 
description of the story of \cite{ticm} in the case of torus actions. A similar narrative 
applies to $\cC_4(G;\mathbf{0})$ and equivariant quantum $K$-theory (minding, however, 
the orbifold nature of $\cC_4$ for general $G$, see \S\ref{coul0}) even though the general 
framework for $K$-theoretic mirror symmetry is incompletely understood.

\begin{remark}\label{lag}
The shortest argument for  the Lagrangian property of $QH^*_G(M)$ passes to the non-commutative 
Coulomb branches of \S\ref{noncomgen}, over which the versions of $QH^*_G(M)$ equivariant under 
loop rotation (which are related to \emph{cyclic homology} of the Fukaya category) 
are naturally modules. The Lagrangian property is now a consequence of the integrability of 
characteristics  \cite{gab} supplemented by finiteness of $QH^*_G(M)$ over $H^*(BG)$. 
\end{remark}
 
\subsection{Coulomb branches with matter.}
The universal property of the $\cC(G;\mathbf{0})$ leaves the spaces $\cC(G;E)$ in search of 
a r\^ole. Their new characterisation addresses this riddle. Namely, the Seidel shift 
operators act on $QH^*(M)$ only when $M$ is \emph{compact}; for more general spaces, 
the most we expect is an action on the \emph{symplectic cohomology}, when the latter is defined 
\cite{ritt}. Equivariant symplectic cohomology $SH^*_G(X)$ is sometimes a localisation of $QH^*_G(X)$, 
in which case the space $\cC_3(G;\mathbf{0})$ will capture a dense open part of $QH^*_G(X)$, with 
portions lost at infinity. Notably, this is the case when $X=M\times V$, with compact $M$ and a 
linear $G$-space $V$. The lost part of $QH^*_G(M\times V)$ can be captured in a second chart of 
$\cC(G;\mathbf{0})$, shifted from the original by the effect of the functor $M\mapsto M\times V$. 

This shift is implemented as follows. The tensor product defines a symmetric monoidal structure on 
$2$-dimensional TQFTs with $G$-gauge symmetry. This structure is mirrored in the 
classifying space $\cC_3(G;\mathbf{0})$ into a \emph{multiplication} along an abelian group  
structure over $\mathrm{Spec}\:H_*^G(\mathrm{point})$. (The latter is isomorphic to the space 
$\frg^{reg}_\bC/G_\bC$ of regular adjoint orbits, and the projection exhibits $\cC_3(G;\mathbf{0})$ 
as a fiberwise group-completion of the classical Toda integrable system; see \S\ref{hopf}.) The 
operation $QH^*_G(M)\rightsquigarrow SH^*_G(M\times V)$ is implemented by  multiplication by a 
certain rational Lagrangian section $\vep_V$ of this group scheme, whose structure sheaf is 
$SH^*_G(V)$. The Lagrangian $\vep_V$ should be regarded as the gauged $B$-model mirror of $V$: see Remark~\ref{superpotentials}.

The precise statement of the main results requires preparation and is postponed to 
\S\ref{stat}; the remainder of this section develops two key examples.  

\subsection{Example I: $G=\mathrm{U}_1$, with the standard representation $L$.} 
\label{keyx}
We have 
\begin{equation}\label{u1}
\cC_3(\mathrm{U}_1;\mathbf{0})= \mathrm{Spec}\,H_*^{\mathrm{U}_1}\!\left(\Omega\mathrm{U}_1;\bC\right) = 
\bC\times \bC^\times \cong T^\vee\bC^\times,
\end{equation}
with $\bC^\times$ \emph{dual} to $\mathrm{U}_1$: the coordinates $\tau$ and $z$ on the two 
factors generate $H^2\left(B\mathrm{U}_1\right)$ and $\pi_1\mathrm{U}_1$. The canonical symplectic 
form $d\tau\wedge dz/z$ also admits an intrinsic topological 
definition, in terms of a natural circle action on $B\mathrm{U}_1\times\Omega\mathrm{U}_1$ (cf.~\S\ref{basiccoul} and \S\ref{noncomu1} below). 

One usually defines the \emph{toric mirror} of the space $L$ as the function (\emph{super-potential}) 
$\psi(z)=z$ on the space $\bC^\times$. The differential $d\psi$ 
defines the Lagrangian $\vep_L:=\{\tau=z\} \subset T^\vee\bC^\times$. View $\vep_L$ instead as 
the rational section  $\tau\mapsto z= \tau$ of the projection $T^\vee\bC^\times\to \bC$ to the 
$\tau$-coordinate, and note in 
passing the Legendre transform $\psi^*(\tau) = \tau(\log\tau -1)$ of $\psi$, in the sense that 
$\vep_L = \exp (d\psi^*)$. 

Functions on $\vep_L$ are identified with $\bC[\tau^\pm]$; this is the $\mathrm{U}_1$-equivariant 
symplectic cohomology of $L$, rather than its quantum cohomology $\bC[\tau]$. We can recover 
the full quantum cohomology by gluing, onto the open set $\tau\neq0$ in \eqref{u1}, a second 
copy $T^\vee\bC^\times$, with coordinates $\tau$ and $z'=z/\tau$. This gluing is compatible with 
projection to the $\tau$-coordinate and leads to the space $\bC^2\setminus\{0\}$, with coordinates 
$(x,y) = (z,\tau/z)$, living over the line $\tau=xy$. The section $\vep_L$ closes now to the 
line $y=1$, identified by projection with the full $\tau$-axis. 

In \cite{bdg, bfn}, $\cC_3(\mathrm{U}_1;L\oplus L^\vee)$ is taken to be the affine completion 
$\bC^2=\mathrm{Spec}\,\bC[x,y]$. The following characterisation is now obvious:

\begin{proposition}\label{ring}
$\bC[x,y]$ is the subring of regular functions $f(\tau,z)$ on 
$T^\vee\bC^\times$ with the property that $f(\tau, z\tau)$ is also regular. \qed
\end{proposition}
 
Our Lagrangian $\vep_L$ is related to the \emph{Euler class of the index bundle} as follows.  
Denote by $\mathfrak{Pic}(\bP^1)$ the moduli stack of holomorphic line bundles on $\bP^1$; its equivariant 
homotopy type is the stack $B\mathrm{U}_1\times\Omega\mathrm{U}_1$ implicit in \eqref{u1}.  
Over $\bP^1\times\mathfrak{Pic}(\bP^1)$ lives the universal line bundle, with fibre the 
standard representation $L$. Its index along $\bP^1$, with a simple vanishing constraint 
at a single marked point, is a virtual bundle  $\mathrm{Ind}_L$ over $\mathfrak{Pic}(\bP^1)$, 
with equivariant Euler class $\mathrm{e}_L \in H^*\left(\mathfrak{Pic}(\bP^1)\right)[\tau^{-1}]$ 
in the localised equivariant cohomology ring. Specifically, $\mathrm{Ind}_L = L^{\oplus n}$ and 
$\mathrm{e}_L=\tau^n$ on the component $\mathfrak{Pic}_n$, $n\in\bZ=\pi_1\mathrm{U}_1$. 
The following is clear from these constructions.

\begin{proposition}\label{euler}
The rational automorphism of multiplication by $\vep_L$ on $T^\vee\bC^\times$, $z\mapsto \tau z$,  
corresponds to the cap-product action of $\mathrm{e}_L$ on $H_*^{\mathrm{U}_1}\!\left
(\Omega\mathrm{U}(1);\bC\right)[\tau^{-1}]$. \qed
\end{proposition} 

These propositions capture the r\^{o}le of $\cC(G;L\oplus L^\vee)$ in quantum cohomology: 
the condition of regularity under capping with the Euler class picks out precisely 
those equivariant Seidel shift operators which act on  $QH^*_{\mathrm{U}_1}(L)$. More 
generally, we have the following

\begin{proposition}\label{seidel}
The subring $\bC[x,y]\subset\bC[\tau,z^\pm]$ acts on $QH^*_{\mathrm{U}_1}(M\times L)$ for any 
compact $\mathrm{U}_1$-Hamiltonian symplectic 
manifold $M$, and it is the largest subring with that property.
\end{proposition}
\begin{proof}
The subring $\bC[\tau]\cong H_*^{\mathrm{U}_1}(\text{point})$ acts in the natural way. 
Recall now (for instance, \cite{op, ir}) that the Seidel element $\sigma_n$ associated with 
$z^n$ (which is a co-character of the original $\mathrm{U}_1$) is the following ``twisted 
$1$-point function": namely, 
the element in $QH^*_{\mathrm{U}_1}(X)$ defined by the evaluation $ev_\infty$ at $\infty$ of 
stable sections of the $X$-bundle over $\bP^1$ associated to  $\cO(-n)$. All is well when $X$ 
is compact: $\sigma_n$ is a unit in $QH^*_{\mathrm{U}_1}(X)$, with inverse $\sigma_{-n}$. 
(Without equivariance, this goes back to Seidel's original paper \cite{seid}.) 
For $X=M\times L$ though, we have a problem when $n<0$: equivariant integration along 
the fibres of $ev_\infty$ incorporates integration along $\mathrm{Ind}_L$, the kernel 
of $H^0\left(\bP^1;\cO(-n) \otimes_{\bC^\times} L\right)\to L$, with dimension $(-n)$; 
the operation contributes its Euler class as a denominator, a factor of $\tau^n$. 
The factor $\tau$ in $y=\tau z^{-1}$ precisely cancels the denominator. \end{proof}
 
\subsection{Generalisation.} 
Propositions~\ref{ring}--\ref{seidel} extend to all $G$ and representations $V$, as 
Theorems~\ref{c3} and~\ref{qh} in \S\ref{stat} below; Theorem~\ref{c4} is the $K$-theory 
analogue. Non-commutative versions of Coulomb branches are described in \S\ref{noncomgen}. 
One required change throughout is the inclusion in the ground ring of an additional 
equivariant parameter $\mu$, from the natural $\bC^\times$-scaling of $V$. The need 
for this will become evident in the example that follows. One can indeed include the 
full $G$-automorphism group of $V$ (the flavor symmetries), but any single scaling 
symmetry that is \emph{compactifying} --- fully expanding or fully contracting --- 
suffices. I will spell out the case of the overall scaling. 
 
\subsection{Example II: $\mathrm{U}_1$ with a general representation $V$.} \label{keyx2}
For a $d$-dimensional representation $V$ of $\mathrm{U}_1$ with weights  $n_1,\dots,n_d\in \bZ$, 
the super-potential $\psi_V:\bC^\times\to \bC$ for its mirror is computed by the following adaptation 
of the Givental-Hori-Vafa recipe.\footnote{The recipe is justified in the SYZ construction by the 
count of holomorphic disks bounding the standard coordinate tori. We are omitting the small quantum 
parameters, one coupled to each coordinate $z_k$.} The defining homomorphism $\rho_V:\mathrm{U}_1\to 
\mathrm{U}_1^d$ of $V$ dualizes to $\rho_V^\vee:(\bC^{\times})^{d}\to \bC^\times$. The standard 
toric super-potential for $\bC^d$ on the source $(\bC^\times)^d$,
\[
\Psi(z_1,\dots,z_d) = z_1+\dots+z_d,
\] 
``pushes down'' to the multi-valued function $\psi_V(z)$ on the target $\bC^\times$ whose 
multi-values are the critical values of $\Psi$ along the fibers of $\rho_V^\vee$. A clean 
restatement is that the Legendre transform $\psi_V^*(\tau)$ is the restriction, under the 
infinitesimal representation $d\rho_V$, of the Legendre transform of $\Psi$: 
in obvious notation, 
\[
\Psi^*(\tau_1,\dots,\tau_d) = \sum\nolimits_k \tau_k(\log\tau_k -1), 
	\quad \psi^*_V = \Psi^*\circ d\rho_V.
\]  

Our Lagrangian $\vep_V$ is the graph of $\exp(d\psi_V^*)$, namely $\tau\mapsto 
z=\prod_k (n_k\tau)^{n_k}$. The reader should meet no difficulty in comparing this 
$\vep_V$ with the Euler class $\mathrm{e}_V$ of the respective index bundle over 
$\mathfrak{Pic}$, as in 
Proposition~\ref{euler}. It should be equally clear how to extend this prescription 
to the case of a higher-rank torus and a general representation.

However, literal application of the lesson from Example~\ref{keyx} runs into trouble, already 
for $\mathrm{U}_1$  with $V =L\oplus L^\vee$. In the GHV construction, the super-potential 
$\Psi=z_1+z_2$ has no critical points along the fibres of $\rho_V (z_{1,2}) = z_1/z_2$. 
We have better luck with the Legendre transform, 
\[
\psi^*_V(\tau) = \tau(\log\tau-1)-\tau(\log(-\tau)-1)= \pi i \tau,
\] 
which identifies $\vep_V$ with the cotangent fibre over $\exp(\pi i)=-1\in \bC^\times$, 
and induces the automorphism 
$z\leftrightarrow (-z)$ of $T^\vee\bC^\times$. While this does match 
Proposition~\ref{euler}, thanks to the Euler class cancellation $\mathrm{e}_{L\oplus L^\vee} 
= \mathrm{e}_{L} \cup\mathrm{e}_{L^\vee} = (-1)^n$ on $\mathfrak{Pic}_n$, raw application of 
Proposition~\ref{ring} would falsely predict that 
$\cC_3(\mathrm{U}_1,V\oplus  V^\vee)=\cC_3(\mathrm{U}_1,\mathbf{0})$,
because $\vep_V$ is now regular.

The remedy incorporates scaling-equivariance into the Euler index class, converting it into 
the $\mu$-homogenized total Chern class. As a Laurent series in $\mu^{-1}$, the latter is defined 
for arbitrary virtual bundles. For the index bundles over $G\ltimes\Omega G$ of representations 
of general compact groups $G$, we will always find \emph{rational functions}. 
With $V=L\oplus L^\vee$, we get $(\mu+\tau)^n(\mu-\tau)^{-n}$ on $\mathfrak{Pic}_n$, 
and the earlier cancellation in the Euler class is now seen to be `fake', arising from 
premature specialisation to $\mu=0$. The Coulomb branch is spelt out in Example~\ref{lplusv}.

Algebraically, $\mu$ is to be treated as an independent parameter. It changes the 
super-potential $\Psi$ by subtracting $\mu\sum\log z_k$; this adds  scale-equivariance 
to the mirror of $\bC^d$. The Legendre transform $\Psi^*$ is modified by the 
substitution $\tau_k\mapsto \tau_k+\mu$, and the topological origin as a scale-equivariant 
promotion of the Chern class is now clearly displayed. 
For a general $V$, the remedied Lagrangian is defined by $z= \prod_k (\mu+n_k\tau)^{n_k}$; 
in particular, it determines the representation. 

Extension to a higher-rank torus, with arbitrary representations, is now a simple matter,  
and it should also be clear how to incorporate the entire \emph{flavor symmetry group} (the 
$G$-automorphism group of $V$), if desired, by equivariant enhancements of the Lie algebra 
coordinates $\tau_k$. There is a characterisation  of $\cC_3$ analogous to 
Proposition~\ref{ring}, as the subring of regular functions on $\cC_3(T;\mathbf{0})$ 
which survive translation by the newly $\mu$-remedied $\vep_V$, and it is easy to  
relate it with the abelian presentations in \cite{n, bdg}. The contribution of 
this paper is the non-abelian generalisation.

\begin{remark}
The remedy of scale-equivariance should not surprise readers versed in 
toric mirror symmetry: na\"ive application of the GHV recipe is  problematic 
for toric actions with non-compact quotients --- which is when our fake cancellations 
can happen --- and the recipe can be corrected by including equivariance under the full torus. 
\end{remark}

\section{Background on Coulomb branches}
\label{coul0}
We recall here the construction and properties of  Coulomb branches; this mostly condenses 
material from \cite{bfm, bf, bfn}. 
I will write $\cC_{3,4}$ for $\cC_{3,4}(G;\mathbf{0})$ when no confusion arises. Denote by 
$H\subset G$ a maximal torus and by $H^\vee, G^\vee$ the Langlands dual groups, $\frg, \frh$ 
the Lie algebras, $W$ the Weyl group.

\subsection{The basic Coulomb branches \cite{bfm}.} \label{basiccoul} 
The space  $\cC_3:= \mathrm{Spec}\,H_*^G(\Omega G;\bC)$ is an affine symplectic resolution of singularities 
of the Weyl quotient $T^\vee H^\vee_\bC/W$. It can be obtained by enlarging the ring of functions on $T^\vee H_\bC^\vee$ 
to include, for all root-coroot pairs $\alpha, \alpha^\vee$ of $G$, those rational functions with principal parts of the 
form $r_\alpha:= (\exp(\alpha^\vee)-1)/\alpha$ along the smooth parts of the root hyperplanes;\footnote{The resulting ring 
includes the functions $r_\alpha$, but is even \emph{larger}, except for the case of $\mathrm{SU}_2$.} this is followed by Weyl division. The $\bC^\times$-action on the cotangent 
fibres arises from the homology grading and scales the symplectic form. The underlying Poisson 
structure is the leading term of a non-commutative deformation over $\bC[h]=H^*(BR)$, 
obtained by incorporating in to $\cC_3$ the equivariance under the loop-rotation circle $R$. 
The loop rotation is revealed by writing $\Omega G\cong LG/G$.

For simply connected $G$, the spectrum of $K^G_*(\Omega G;\bC)$ is also a symplectic manifold, which gives an affine resolution 
of the periodicized cotangent bundle $(H_\bC\times H^\vee_\bC)/W$. This is now similarly accomplished by adjoining rational functions 
with principal parts $(\exp(\alpha^\vee)-1)/(\exp(\alpha)-1)$ along the smooth parts of the hypertori $\exp(\alpha)=1$, 
before Weyl division. However, this space has singularities 
when $\pi_1 G$ has torsion. Write $G=\tilde{G}/\pi$ for the torsion subgroup $\pi\subset\pi_1G$, $H=\tilde{H}/\pi$. 
As a subgroup of $ Z(\tilde{G})$, $\pi$ acts by automorphisms of $K_{\tilde{G}}(X)\otimes\bC$ for 
any $G$-space $X$: to see this, decompose a class in $K_{\tilde{G}}(X)$ into $\pi$-eigen-bundles, 
and multiply each of them by the corresponding character of $\pi$, before re-summing to a complex 
$K$-class. 
We adopt the smooth symplectic orbifold $\pi\ltimes\mathrm{Spec}\, K^{\tilde{G}}_*(\Omega G;\bC)$ 
as the definition of $\cC_4$.

\begin{remark}[Sphere topology]\label{sphere}
Some features of $\cC_{3,4}$ are explained by \emph{Chas-Sullivan theory} in dimension $3$, one 
higher than usual. The underlying topological object is the mapping space from $S^2$ to the 
stack $BG$; it has a natural $E_3$ structure, which turns out to correspond to the Poisson 
form on $\cC_{3,4}(G;\mathbf{0})$. 
Loop rotation is  seen in the presentation as the two-sided groupoid $G\ltimes LG \rtimes G$, 
with Hecke-style product (see Remark~\ref{hecke} below). Tracking the loop rotation breaks $E_3$ 
down to $E_1$, 
because rotating spheres in an ambient $\bR^3$ may be strung together linearly as beads on
the rotation axis, but can 
no longer move around each other. This leads to the non-commutative Coulomb branches we 
shall review in \S\ref{noncomgen}.
\end{remark}

\subsection{Group scheme structure.}\label{hopf}
 The Hopf algebra structures of $H_*^G(\Omega G), K^G_*(\Omega G)$
over the ground rings $H_G^*, K_{G}$ of a point lead to relative abelian group structures 
\begin{equation}\label{intsyst}
\cC_3(G;\mathbf{0})\xrightarrow{\ \chi\ }\frh_\bC/W, \qquad 
	\cC_4(G;\mathbf{0})\xrightarrow{\ \kappa\ }\pi\ltimes (\tilde{H}_\bC/W).
\end{equation}
When $\pi_1G$ has torsion, the second base is an affine orbifold whose ring of functions 
is $K_G(\mathrm{point})$. (The abelian property is a piece of characteristic-zero good 
fortune: the correct commutativity structure is $E_3$, as explained in Remark~\ref{sphere}, 
but this decouples into a strictly commutative and a graded Poisson structure.)
These maps define integrable systems: 
$\chi$ is a partial completion of the classical Toda system\footnote{This was 
rediscovered in \cite{ticm}; I thank H.~Nakajima for pointing 
me to the original reference.} \cite{bf}, whereas $\kappa$ is its finite-difference version.

\begin{remark}[Adjoint and Whittaker descriptions]\label{whittaker}
As an algebraic symplectic manifold, $\cC_3$ is the algebraic symplectic reduction 
$T^\vee_{reg} G_\bC^\vee/\!/G_\bC^\vee$ of the fibrewise-regular part of the cotangent 
bundle under conjugation. There is a simliar description of $\cC_4$ using the 
Langlands dual Kac-Moody group (\emph{not} the loop group of $G^\vee$), 
capturing the holomorphic (but not algebraic) symplectic structure. 

The space $\cC_3$ has another description as the two-sided symplectic reduction 
of $T^\vee G_\bC^\vee$ by $N$, at the regular nilpotent character. Clearly, this is 
algebraic symplectic; much less obviously, it is hyper-K\"ahler, thanks to work of Bielawski 
on the Nahm equation \cite{bi}. 
The non-commutative deformation has a corresponding description in terms of $N\times N$ 
monodromic differential operators on $G^\vee_\bC$ \cite{bf}. 

In both descriptions, multiplication along the group $G^\vee$ induces the group scheme 
structure of \S\ref{hopf}. Commutativity is more evident in the adjoint description, where 
the Toda fibers are the centralizers of regular co-adjoint orbits in $\frg_\bC^\vee$.
\end{remark}

\subsection{Coulomb branches for $E=V\oplus V^\vee$.} \label{LV}
To build the spaces $\cC_{3,4}(G;E)$, we follow  \cite{bfn}, to which we refer for full details,  
and replace $\Omega G$ in the original $\cC$ by a \emph{linear space} $L_V\to\Omega^{a} G$, a stratified 
space whose fibres are vector bundles over the Schubert strata of the algebraic model 
$\Omega^{a}G:=G_\bC(\!(z)\!)/G_\bC[\![z]\!]$ of $\Omega G$. The fibre of $L_V$ over a Laurent 
loop $\gamma\in \Omega^{a} G$ is the kernel of the difference 
\begin{equation}\label{kay}
 \left.L_V\right|_\gamma \to V[\![z]\!] \oplus V[\![z]\!] \xrightarrow{\mathrm{Id}-\gamma} V(\!(z)\!).
\end{equation}
Projection embeds $L_V$ in either factor $V[\![z]\!]$ with finite co-dimension, which is bounded 
on any finite union of strata in $\Omega^a G$. More precisely, the complex \eqref{kay} descends to 
$G[\![z]\!]\backslash \Omega^aG$, with the left and right copies of $G[\![z]\!]$ acting on the 
respective factors $V[\![z]\!]$, and the left one alone acting on $V(\!(z)\!)$. Over any finite 
union of strata, $L_V$ contains two sub-bundles of finite co-dimension, coming from a left and a 
right $z^nV[\![z]\!]$, for sufficiently large $n$. This stratified finiteness lets one define 
the Borel-Moore ($K$-)homologies $BMH^G_*(L_V), BMK^G_*(L_V)$, renormalising the grading as 
if $\dim V[\![z]\!]$ were zero. 

The normalised grading is compatible with the multiplication defined by the following correspondence 
diagram on the fibres of $L_V$, which lives over the multiplication of two loops $\gamma,\delta\in\Omega^a G$:
\begin{equation}\label{correspondence}
\left.L_V\right|_\gamma \oplus  \left.L_V\right|_\delta \quad\leftarrowtail\quad
	\left.L_V\right|_\gamma {\underset{V[\![z]\!]}{\oplus}}  \left.L_V\right|_\delta
 	\quad\rightarrowtail\quad\left.L_V\right|_{\gamma\cdot\delta};
\end{equation}
the sum in the middle is fibered over the right component of $\left.L_V\right|_\gamma$ and the left one of 
$\left.L_V\right|_\delta$, while the right embedding is the projection to the outer $V[\![z]\!]$ summands.
The wrong-way map in homology along the first inclusion is well-defined, over $\gamma,\delta$ in a finite 
range of Schubert cells, after modding out by a common subspace $z^nV[\![z]\!]$, and the result is 
independent of $n$. 

As before, non-commutative deformations arise by including the loop  rotation $R$-action 
on $\Omega G$ and on $V[\![z]\!]$; their leading terms define Poisson structures.

\begin{remark}[$E_3$ Hecke property]\label{hecke}
A Laurent loop defines a transition function for a principal $G_\bC$-bundle over the 
non-separated disk $\ddisk$ with doubled origin. The multiplications have a Hecke interpretation 
as correspondences on $G\ltimes\Omega G$ and $L_V$, induced by following left-to-right the maps 
relating non-separated disks with doubled and tripled centres:
\[
\genfrac{}{}{0pt}{}{(\ddisk)}{(\ddisk)} \quad\overset{g}{-\!\!\!\twoheadrightarrow} \quad(-\raisebox{-1.5pt}{\vdots}-) 
	\quad\overset{i}{\leftarrowtail}\quad (\ddisk)
\]
The map $g$ glues the bottom sheet of the first disk to the top sheet of the second, while $i$ 
hits the outer centres of the triple-centred disk. The $E_3$ property comes from sliding the multiple 
centres around, as in Chas-Sullivan sphere topology. With rotation-equivariance, this freedom is lost 
and we are reduced to an $E_1$ multiplication.

On $G$-bundles, the Hecke operation is represented by multiplication of transition 
functions, once we identify, on the left side, the top bundle on its bottom sheet with the bottom bundle 
on its top sheet. Next, associated to the representation $V$ is a vector bundle over $\ddisk$, whose 
 space of global sections is $L_V$. The correspondence \eqref{correspondence} arises by retaining 
those pairs of global sections on the left which match on the glued pair of sheets, and then 
restricting them to the top and bottom sheets of the triple-centred disk. 
\end{remark}

\subsection{Massive versions.}\label{scale}
We enhance the Coulomb branches by the addition of a symmetry in which $\bC^\times\supset S^1$ 
scales the fibres of $L_V$: 
\[\begin{split}
\cC^\circ_3(G;E)&:=  \mathrm{Spec}\,BMH_*^{G\times S^1}(L_V;\bC),\qquad\text{ projecting to } \frh_\bC/W\times \bC, \\ 
	\cC^\circ_4(G;E) &: = \pi\ltimes\mathrm{Spec}\,BMK_*^{\tilde{G}\times S^1}(L_V;\bC),
		\text{ projecting to } \pi\ltimes (\tilde{H}_\bC/W) \times\bC^\times. 
\end{split}\]
The projections to the massive Toda bases are defined as in \eqref{intsyst}, and denoted  
by $\chi(\mu),  \kappa(m)$, with generators $\mu\in H^2(BS^1), m^\pm\in K_{S^1}(\mathrm{point})$. 
The fibres over fixed values of the parameters $\mu,m$ are total spaces of (usually singular) 
integrable systems; this will follow from flatness of the projections to the Toda bases. 
The scaling  is trivial when $E=\mathbf{0}$ and $L_V=\Omega^aG$, but it will couple to the 
Euler class of the index bundle over $G\ltimes\Omega G$, promoting it to the total Chern class.

The notation is subtly abusive: the $\cC^\circ$ depend on the polarisation $V$ and not just on $E$. 
For instance, switching $V\leftrightarrow V^\vee$ leads an isomorphic space only if we also change 
the orientation of the rotating circle. This $V$-dependence disappears at $\mu=0$ or $m=1$. 
We will see in \S\ref{proofs} that the $\cC^\circ(G;E)$ are flat over $\bC[\mu], 
\bC[m^\pm]$, and that the same spaces $\cC_{3,4}(G;E)$, as defined earlier in 
this section, appear by specialising to $\mu=0$ or $m=1$, independently of the choice of $V$.

\section{Main results}
\label{stat}
We are finally in position to state Theorems~1--3; the non-commutative analogues 
of Theorems~1 and~2 will wait until \S\ref{noncomgen}. First, I describe the Lagrangians 
generalising the massive $\vep_V$ of Example~\ref{keyx}. Their Euler class interpretation, 
already  mentioned following Proposition~\ref{ring},  will be  spelt out in \S\ref{proofs} 
below. 
  
\subsection{The Euler Lagrangians.}\label{eulag}
For $w\in \bC^\times$ and $\nu$ a weight of $H$, $w^\nu:= \exp(\nu\log w)$ determines a point 
in $H_\bC^\vee$. Consider the following rational maps from $\frh_\bC\times\bC$ 
and $H_\bC\times\bC^\times$ to $H_\bC^\vee$, 
defined in terms of the weights $\nu$ of $V$, which are to be included with their multiplicities:
\begin{equation}\label{eulerdef}
\vep_V:(\xi,\mu) \mapsto \prod\nolimits_\nu \left(\mu+\langle\nu|\xi\rangle\right)^\nu, \qquad
	\lambda_V:(x,m)\mapsto\prod\nolimits_{\nu} \left(1-(mx^\nu)^{-1}\right)^\nu.
\end{equation}
(In parsing each formula, note the double use of $\nu$, first as infinitesimal character of 
$H$ and then as co-character of $H^\vee$.) 
The maps are Weyl-equivariant and their graphs are regular, away from a co-dimension~$2$ 
locus over their domains (cf.~\S\ref{codim2} below); their closures define Lagrangian 
sub-varieties $\bar{\vep}_V\subset \cC^\circ_3(G;\mathbf{0})$ and $\bar{\lambda}_V\subset 
\cC^\circ_4(G;\mathbf{0})$ over their respective ground rings $\bC[\mu], \bC[m^\pm]$. 

\begin{remark}[Broader picture]\label{superpotentials}
For generic $\mu$ and $m$ (but most meaningfully, near $\mu, m=\infty$), the maps \eqref{eulerdef} are 
the exponentiated differentials of the following functions, in which $\xi\in \frg_\bC$ and 
$x\in G_\bC$ are the arguments while $\mu, m$ are treated as parameters:
\[
\xi\mapsto\mathrm{Tr}_V \left[(\xi\oplus\mu)\cdot (\log(\xi\oplus\mu) - 1)\right], \qquad 
		x\mapsto \mathrm{Tr}_V\,\mathrm{Li}_2((x\times m)^{-1}).
\]
The first function appeared as the ``$\Sigma\log\Sigma$ Landau-Ginzburg $B$-model 
mirror'' of the abelian GLSM on $V$: \cite{w}, and see also Remark~\ref{stirling}. 
The Lagrangian $\lambda_V$ and its primitive appeared\footnote{For the adjoint 
representation, but the discussion in \emph{loc.~cit.} applies to any $V$.} 
in the index formula for K\"ahler differentials over the moduli of $G$-bundles on curves 
\cite[Eqn.~6.2 and Thm.~6.4]{tw}, with the powers of $m^{-1}$ tracking the degree of the forms. 
The relation with Coulomb branches was not known at the time. Today, we would express that 
index formula in terms of Lagrangian calculus in $\cC_4^\circ(G,\mathbf{0})$, namely the intersection 
of $\lambda_V$ with the graphs of certain isogenies $H_\bC\to H_\bC^\vee$, defined from the 
\emph{levels} of central extensions of the loop group $LG$. 
Those isogenies correspond to the Theta line bundles on the moduli of $G_\bC$-bundles on curves; 
they are semiclassical limits of Theta-functions --- in the same sense that the Lagrangians 
$\vep_V,\lambda_V$ are semiclassical $\Gamma$-functions, see \S\ref{noncomgen} --- and are also twists 
of the unit section by the discrete Toda Hamiltonian of $\cC_4$. 
\end{remark}

\subsection{Algebraic description of the Coulomb branches.} 
The first two results generalise to non-Abelian $G$ the explicit presentations of Coulomb 
branches given in \cite{n, bdg} for torus groups. Their proofs, in \S\ref{proofs}, are 
straightforward; 
more intriguing are the the non-commutative generalisations in \S\ref{noncomgen}. To state 
the theorems, note that translation on the group schemes by the section $\vep_V$, respectively 
$\lambda_V$, gives a rational symplectomorphism of $\cC^\circ_{3}, \cC^\circ_{4}$, relative 
to the massive Toda projection of \S\ref{scale}.

\begin{maintheorem}\label{c3}
The space $\cC^\circ_3(G;E)\to \frh_\bC/W\times \bC$ is the affinisation of two copies of 
$\cC^\circ_3(G;\mathbf{0})$ glued together by means of $\vep_V$-translation. In other words: 
regular functions on $\cC^\circ(G;E)$ are those regular functions on 
$\cC^\circ(G;\mathbf{0})$ which remain regular after translation by $\vep_V$.
\end{maintheorem}

\begin{maintheorem}\label{c4}
The orbifold $\cC^\circ_4(G;E)\to \pi\ltimes (\tilde{H}_\bC/W) 
\times\bC^\times$ is the relative affinisation of two copies of $\cC^\circ_4(G;\mathbf{0})$ 
glued together by means of $\lambda_V$-translation.
\end{maintheorem}
\noindent
Abstractly, the spaces are the quotients, in affine schemes over the massive Toda bases, of an 
equivalence relation on $\cC^\circ{\scriptstyle \coprod} \cC^\circ$ defined from $\bar{\vep}_V, 
\bar{\lambda}_V$. The relation is not very healthy, being neither proper nor open. 
The surviving condition can equally well be imposed  
prior to Weyl division, and can be restated in terms of growth constraints along the Toda fibres over the smooth parts of the loci of zeroes and 
poles of $\vep_V,\lambda_V$ (\S\ref{codim2} below). 
Meanwhile, the next theorem, characterising the regular functions on $\cC^\circ(G;E)$ in terms of quantum cohomology, is simple 
enough to prove here.

\begin{maintheorem}\label{qh}
$\bC[\cC^\circ_3(G;E)]$ comprises those functions on $\cC^\circ_3$ which act regularly 
on the equivariant quantum cohomologies $QH^*_{G\times S^1}(M\times V)$, for compact 
Hamiltonian $G$-manifolds $M$. \\
$\bC[\cC^\circ_4(G;E)]$ comprises those regular functions on $\cC^\circ_4$ which act on the 
equivariant quantum $K$-theories $QK^*_{G\times S^1}(M\times V)$, for compact Hamiltonian $G$-manifolds $M$. 
\end{maintheorem}

\begin{proof}[Proof of Theorem~3]
Away from the root hyperplanes on the massive Toda base (or the singular conjugacy locus, 
respectively), the statement follows by abelianisation from the calculation 
of Proposition~\ref{seidel}. On the other hand, away from $\mu=0$ (or $m=1$), the fixed-point 
theorem allows us to ignore $E$ and $V$, and we are 
reduced to the action of $H_*^G(\Omega G)$ on equivariant quantum cohomology (see \cite{ticm, tprin}). 
The remaining locus has co-dimension $2$ on the base, over which $QH^*_{G\times S^1}(M)$ is 
finite and free as a module.
\end{proof}

\section{Some consequences}
We discuss briefly some geometry of the Coulomb branches as it emerges from their description 
in \S\ref{stat}. Flatness and normality were already established 
in \cite{bfn}; we will review them in the new construction.  

\subsection{Generic geometry of the Coulomb branches.}\label{codim2}
The divisor $S$ of singularities of the section $\vep_V$, resp.\ $\lambda_V$ is the unions of 
hyperplanes $S_\nu$ defined by the monomial factors in \eqref{eulerdef}. The pairwise 
intersections of the $S_\nu$ contain the indeterminacy locus $I$. Away from $I$, each  
$\cC^\circ(G;E)$ is the affinisation of a smooth space, obtained by gluing two open charts 
$\cC^\circ$ with a vertical relative shift over the Toda base. Away from $S$, 
the glued space is of course isomorphic to the original $\cC^\circ$; whereas, 
near each $S_\nu\setminus I$, the Toda fibres undergo a nodal degeneration 
along the $\bC^\times$ factor $\bC^\nu$, modeled on  $\bC^\times \rightsquigarrow 
\bC\sqcup_0 \bC$ in the fibres of the $A_{n-1}$-singularity $(x,y)\mapsto t = (xy)^{1/n}$. 
(The number $n$ is computed from the divisibility and the multiplicities of the weight $\nu$.) 
The appearance of the nodal locus, 
along which $\cC^\circ(G;E)$ is singular when $n>1$, is a consequence of affinisation: 
the smooth charts $\cC^\circ$ cover the complement, as in Example~\ref{keyx}.
From here, Hartog's theorem determines  $\cC^\circ(G;E)$ completely; but we can be more 
specific in concrete cases. Thus, some fibres of $\cC^\circ(G;\mathbf{0})$ are crushed 
in co-dimension~$2$, over $I$.

\subsection{Example: $\mathrm{U}_1$ with $L\oplus L^\vee$.}\label{lplusv}
The space $\cC_3^\circ(\mathrm{U}_1; L\oplus L^\vee)$ is the quadric cone $xy=\mu^2-\tau^2$. 
In the original coordinates $\{\tau, z^\pm, \mu\}$, the rational automorphism 
$z\mapsto z(\mu+\tau)/(\mu-\tau)$ preserves precisely the subring generated by  
$\mu, \tau, x = (\mu-\tau)z, y = (\mu+\tau)z^{-1}$. The two copies of $\cC^\circ_3$ 
map to the constructible subsets 
\[
\begin{split}
\{\mu^2\neq\tau^2\}\sqcup \{\mu=\tau, y\neq0\}\sqcup\{\mu=-\tau, x\neq0\}\sqcup\{0\} 
\\
\{\mu^2\neq\tau^2\}\sqcup\{\mu=\tau, x\neq0\}\sqcup \{\mu=-\tau, y\neq0\}\sqcup\{0\}
\end{split}
\] 
whose union misses the nodal lines $x=y=0$ in the fibres over $\mu=\tau$ and $\mu=-\tau$, 
with the exception of their intersection at the vertex $0$, onto which the zero-fibre of 
each $\cC^\circ_3$ gets crushed. 

\subsection{Example: $\mathrm{SU}_2$ with the standard representation.} \label{su2std}
Consider the Weyl double cover of $\widetilde{\cC}^\circ_3$ of $\cC^\circ_3$, the $\mathrm{Spec}$ of $\bC[\mu, z^{\pm},\tau,\frac{z-1}{\tau}]$, in the $z,\tau$-notation already used for 
the maximal torus of $\mathrm{SU}_2$. The 
functions over $\widetilde{\cC}^\circ_3$ are generated over 
$\bC[\mu,\tau]$ by  $u=\frac{z-1}{\tau}$ and $v=\frac{1-1/z}{\tau}$, with the single 
relation $u - v = \tau uv$. The Weyl action switches $u$ and $v$ and changes the sign of $\tau$. 
Translation by $\vep_V$ sends $z$ to $\frac{\mu+\tau}{\mu-\tau}\: z$. 
Let $x: = \mu u - z, y:= \mu v -z^{-1}, w = (x-y)/\tau$; the surviving subring is described 
by generators and relations over $\bC[\mu,\tau]$ as 
\[
\left\{x, y, w\right\}, 
	\text{ with relations }x-y = \tau w ,\quad xy = 1+ \mu w. 
\]
(We justify the generators in the next example.) Setting $\mu=0$  yields the 
ring $\bC[\tau, z^\pm, \frac{z-1/z}{\tau}]$. This is $\bC\times \bC^\times$, with 
the points $(0,\pm1)$ blown up and the proper transform of $\tau=0$ removed. Each 
of the two $\widetilde{\cC}^\circ_3$ charts covers one of the exceptional divisors 
and misses the other. 

\subsection{Example: $\mathrm{SU}_2$ with a general representation.} \label{su2}
Factor $\vep_V(\mu,\tau) = \phi(\mu,\tau)\phi^{-1}(\mu,-\tau)=\phi_+\phi_-^{-1}$, 
with a homogeneous polyonomial $\phi$ of degree $N$, and let $x=(z\phi_--\mu^N)/\tau, y =(\mu^N-z^{-1}\phi_+)/\tau$ and $w=(x-y)/\tau$ as before. 
Generators and relations for the surviving subring are
\begin{equation}\label{gens}
\left\{x,y, w\right\}, \text{ with }  x-y = \tau w ,\quad
xy = \frac{\mu^{2N} -\phi_+\phi_-}{\tau^2} + 	\mu^N w.
\end{equation}
Setting $\mu=0$ gives the subring generated by $\tau^{N-1} 
(z-(-1)^Nz^{-1}) $ and $\tau^{N-2}(z+(-1)^Nz^{-1})$. This reproduces 
the result of \cite[Example 6.9]{bfn}.

For instance, choosing the adjoint representation gives $N=2$ and the Weyl invariant ring 
is $\bC[\tau^2, z+z^{-1}, \tau(z-z^{-1})]$, defining the quotient $T^\vee \bC^\times/\{\pm1\}$. This is the Coulomb branch for the zero representation of $\mathrm{U}_1$, Weyl quotiented by ${\pm1}$. More generally, any representation with $N>1$ leads to the Weyl quotient of the $\mathrm{U}_1$ Coulomb branch for a 
representation with an $N$ that is lower by $2$ --- such as $V\ominus \frg/\frh$, if $V$ happened 
to contain the adjoint representation. We generalize this in the Appendix.

\subsection{Checking the $\mathrm{SU}_2$ example.} Let $A$ be the surviving subring, and 
$A'\subset A$ the subring generated by \eqref{gens}; let us check that $A'=A$. This is clear 
with $\tau$ inverted, by reduction to the case of $\mathrm{U}_1$, when $z\phi_-$ and $z^{-1}\phi_+$ 
generate $\widetilde{\cC}^\circ_3$ over $\bC[\mu, \tau^\pm]$. Upon formal completion near $\tau=0$, 
the statement is equally clear with $\mu$ inverted, when $\phi_\pm$ become units. This shows that 
$A/A'$ is a quasi-coherent torsion sheaf on the $(\mu,\tau)$-plane  supported at $\mu=\tau=0$. 
But such a sheaf would yield a $\mathrm{Tor}_2$ group against the sky-scraper at $\mu=\tau=0$, 
which is forbidden, because (I claim) both $A'$ and $A$ are flat over $\bC[\tau,\mu]$. 
Flatness $A'$ is checked easily from the $3$-step resolution built from \eqref{gens}; that of $A$ is discussed below.

\subsection{Normality.}
Our description of $\cC^\circ(G;E)$ implies its normality: indeed, if a function $f$ is integral 
over the surviving subring, then $f\circ(\vep_V\cdot)$ is integral over $\cC^\circ$, so it is 
regular, and so $f$ survives. Alternatively, granting flatness of the Toda projections, one sees the desired regularity in co-dimension~$1$ from the generic geometric 
behaviour described in \S\ref{codim2}. 

Flatness of the Coulomb branches over the massive Toda bases (freedom, in fact) is 
wrapped into the proof of Theorem~1, in the next section. Normality of the massless specialisation follows from flatness, as we can again check regularity in co-dimension $1$: 
the generic Abelian description applies away from the root hyperplanes, while on the generic part of 
a root hyperplane the $\mathrm{SU}_2$ description of  Example~\ref{su2} takes its place.

\section{Proof of Theorems~1 and 2}\label{proofs}
We use the Schubert stratification of $\Omega^a G$ into $G[\![z]\!]$-orbits. Even-dimensionality 
collapses the associated spectral sequences and leads to ascending filtrations on the rings 
$\bC[\cC^\circ(G;E)]$. 
The associated graded components are easily described (\S\ref{assgr} below), and are locally 
free over the Toda bases. This makes the original rings locally free as well; in particular, 
they are flat over $\bC[\mu], \bC[m^\pm]$.

I write out the proof for $\cC^\circ_3$; the $K$-theory case is entirely parallel. Call $A_V$ 
the ring, implied in Theorem~\ref{c3}, of regular functions on $\cC^\circ_3$ which survive 
$\vep_V$-translation. We will see from topology how this last operation is compatible with the 
Schubert filtration, so that we can also define the subring  $\Sigma_V\subset \mathrm{Gr}\,
\bC[\cC^\circ_3]$ of symbols which remain regular after $\vep_V$-translation. 
Clearly, $\mathrm{Gr}A_V\subset \Sigma_V$. The theorems will follow from two observations:
\begin{enumerate}\itemsep0ex
\item $\bC[\cC^\circ_3(G;E)]\subset A_V$;
\item $\mathrm{Gr}\,\bC[\cC^\circ_3(G;E)] = \Sigma_V$.
\end{enumerate} 

\subsection{The index bundle.}
Over the stack $\mathfrak{Bun}_G(\bP^1)$  of principal $G_\bC$-bundles over $\bP^1$, there lives 
the virtual index bundle $\mathrm{Ind}_V$, the holomorphic Euler characteristic of the sheaf of 
sections of $V$ over $\bP^1$ with simple vanishing condition at one marked point $\infty$. 
It is a class in $K_{G\times S^1}^0(\Omega G)$, after incorporating the mass parameter $\mu$ 
(equivariance under scaling of $V$). Call  $\mathrm{e}_V$ its equivariant Euler class, 
more accurately defined as the $\mu$-homogenised $G$-equivariant total Chern class of 
$\mathrm{Ind}_V$. The following two propositions are understood after suitable localisation 
on the massive Toda base  $\frh_\bC/W_\bC\times\bC$. 

\begin{proposition}\label{cap}
Translation by $\vep_V$ on $\cC_3$ corresponds to cap-product with $\mathrm{e}_V$ on 
$H_*^{G\times S^1}(\Omega G)$. 
\end{proposition}   
\begin{remark}\label{group}
Cap-product with $\mathrm{e}_V$ \emph{must a priori} correspond to translation by some rational 
section: the index bundle is additive for the sphere multiplication in $G\ltimes\Omega G$, 
so its Euler class is multiplicative. As a group-like element in the dual Hopf algebra, 
it represents a (rational) section of the group scheme $\cC^\circ$ over its Toda base. 
We identify this section by abelianization.
\end{remark}
\begin{proof}
Localise to the complement of the root hyperplanes on the Toda base to reduce, by the 
fixed-point theorem, to the case of a torus, where Proposition~\ref{euler} 
applies (as enhanced in Example~\ref{keyx2}). 
\end{proof}

\begin{corollary}
The Schubert filtration is preserved by $\vep_V$-translation. \qed
\end{corollary}

\subsection{Two embeddings of $\bC[\cC^\circ_3(G;E)]$.}\label{twoemb} 
Refer to the notation in \S\ref{LV} and Remark~\ref{hecke}. The Hecke construction 
at $0\in \bP^1$ maps the stack $\mathfrak{Bun}_G(\ddisk) = G[\![z]\!]\backslash\Omega^aG$ 
of $G_\bC$-bundles over the double-centered disk to  $\mathfrak{Bun}_G(\bP^1)$. This gives 
an equivariant homotopy equivalence and in particular a ($K$-)homology equivalence. 
The key observation is that, 
restricted to $G[\![z]\!]\backslash\Omega^aG$, 
$\mathrm{Ind}_V$ is the ``left minus  right" copy of $V[\![z]\!]$. 

More precisely, note the two inclusions $\iota_{l,r}:L_V\hookrightarrow V[\![z]\!]$, and recall 
that over any finite union of strata, $L_V$ contains a finite co-dimension sub-bundle. Quotienting 
it out  regularises the difference of $V[\![z]\!]$-bundles into a class in $K_{G\times S^1}(\Omega G)$. 
A moment's thought identifies this with $\mathrm{Ind}_V$, as the index of the Hecke transform of 
the trivial $V$-bundle on $\bP^1$, minus that of the trivial $V$-bundle.

Each inclusion $\iota_{l,r}$ defines a graded ring homomorphism $\varphi_{l,r}:\bC[\cC^\circ_3(G;E)]\rightarrowtail\bC[\cC^\circ_3]$, by intersecting with the zero-section 
in the ambient bundle. Per 
our discussion, $\varphi_l = \mathrm{e}_V\cap \varphi_r$. Using $\varphi_r$ to pin 
down $\cC^\circ_3(G;E)$, Proposition~\ref{cap} now settles Observation (i).

\subsection{Working out $\Sigma_V$.} \label{indexsplit}
For a $1$-parameter subgroup $z^\eta\in \Omega H$, with Schubert 
stratum $C_\eta$ and  Levi centraliser $Z(\eta)\subset G$, split $V = V_+\oplus V_0\oplus V_-$ 
following the sign of the $\eta$-eigenvalue. The index bundle then splits as $\mathrm{Ind}_V 
=I_+(\eta)\ominus I_-(\eta)$,  with the $\nu$-weight space of $V_\pm$ appearing 
$\pm\langle\nu|\eta\rangle$ times in $I_\pm(\eta)$. The Euler class $\mathrm{e}_V$ factors 
at $z^\eta$ as 
\[
\left.\mathrm{e}_V\right|_{z^\eta} = 
\mathrm{e}_+(\eta) \cdot\mathrm{e}_-(\eta)^{-1}, \text{ with }\mathrm{e}_\pm(\eta) = 
\prod_\nu(\mu+\nu)^{|\langle\nu|\eta\rangle|}.
\] 
There is a (degree-shifting) isomorphism
\[
\mathrm{Gr}_\eta\bC[\cC^\circ(G;\mathbf{0})] = BMH_*^{G\times S^1}(C_\eta) \cong 
       H_*^{Z(\eta)\times S^1}(\mathrm{point}),
\]
and the $\eta$-graded component of $\Sigma_V$ is the subspace $\mathrm{e}_-\cap 
\mathrm{Gr}_\eta\bC[\cC^\circ(G;\mathbf{0})]$.

\subsection{Working out $\mathrm{Gr}\,\bC[\cC^\circ_3(G;E)]$.} \label{assgr}
Collapse of the Schubert spectral sequence implies that
\[
\mathrm{Gr}_\eta\bC[\cC^\circ_3(G;E)] = BMH_*^{G\times S^1}(\left.L_V\right|_\eta).
\]
Now, the the homology group is generated over $H_*^{Z(\eta)\times S^1}(\mathrm{point})$ by the 
fundametal class of the total space of $L_V$ over $C_\eta$, whose complement in the right 
$V[\![z]\!]$ of \eqref{kay} is precisely $I_-(\eta)$; therefore 
\[
\mathrm{Gr}_\eta\bC[\cC^\circ_3(G;E)] = \mathrm{e}_-(\eta)\cap \mathrm{Gr}_\eta\bC[\cC^\circ(G;\mathbf{0})],
\]
in agreement with the $\eta$-component of $\Sigma_V$ above. This settles Observation~(ii).

\section{Non-commutative Coulomb branches} 
\label{noncomgen}

Incorporating the loop rotation circle $R$ in the previous constructions leads to 
non-commutative deformations  
$\cN\cC^\circ_{3,4}(G;E)$ of the Coulomb branches over the ground rings $\bC[h]=H^*(BR)$ and 
$\bC[q^\pm]=K_R(\mathrm{point})$, respectively. The geometric objects exist in the formal 
neighbourhoods of $h=0$ and $q=1$; away, only their function rings $\cA_{3,4}$ survive. 
Nonetheless, we sometimes keep the convenient conversational pretence of underlying 
spaces $\cN\cC$. The calculation in \S\ref{proofs} for their description applies with 
only minor changes: we are only missing the good statements, which we summarise below 
before spelling out the argument. 

This section is rather sketchy; a development spelling out the r\^ole of our 
non-commutative solutions, the $\Gamma$-functions, in connection with the 
GLSM, is planned for a follow-up paper. 

\subsection{Summary.} The integrable abelian group structure of the $\cC^\circ$ over their Toda 
bases deforms to a symmetric tensor structure\footnote{I thank David Ben-Zvi for pointing 
out to me the generality of this statement.} on $\cA$-modules, induced from the diagonal 
inclusion $\Omega G\rightarrowtail \Omega G\times\Omega G$. Restricting the module structure
to the Toda base, this is the ordinary tensor product, with tensor unit the structure sheaf 
$\cO_1$ of the identity section. For $\cC_3^\circ$  in the Whittaker presentation 
(Remark~\ref{whittaker}), the operation comes 
from  convolution of $\mathscr{D}$-modules on the Langlands dual group $G^\vee$: from this 
stance, the symmetric monoidal structure is developed in \cite{bzg}.

The Lagrangians $\vep_V,\lambda_V$ deform to modules $E_V, \Lambda_V$ over $\cA_{3,4}$, and 
the (rational) automorphisms of $\cC$ defined by $\vep_V,\lambda_V$-translation become, 
on $\cA$-modules, the functors of convolution with $E_V, \Lambda_V$. The Hamiltonian nature 
of the translations renders these functors (generically) trivialisable by (singular) inner 
automorphisms of $\cA$.  
In Theorem~\ref{cnc}, I characterise the Coulomb branches $\cN\cC^\circ(G;E)$ as the subrings of 
elements of $\cA$ which survive these inner automorphisms (that is, remain regular).

While this loose description of the $\cN\cC^\circ(G;E)$ appears uniform, a 
distinction arises between formal and genuine deformations. Formally, the modules 
$E_V$ and $\Lambda_V$ are generically invertible, analogous to flat line bundles with 
singularities, with the latter located on the singular loci of the sections 
$\vep_V,\lambda_V$. If, following 
the language of $\mathscr{D}$-modules, we call \emph{solutions} the $\cA$-module 
morphisms to the identity section $\cO_1$, then the super-potentials that were introduced in 
Remark~\ref{superpotentials} are the leading $h\to 0$ asymptotics of the 
logarithms of the solutions (cf.~Remark~\ref{stirling}).

With the deformation parameters turned on, these asymptotics become meromorphic 
solutions that are easily found. For $E_V$ on $\cC_3$, a solution is the 
$\Gamma$-function of the 
representation $V$ (recalled in \S\ref{gamma} below), while a $q$-analogue solves 
$\Lambda_V$ on $\cC_4$. Conjugation by these solutions impose the defining regularity 
constraints for the $\cN\cC^\circ(G;E)$. Outside the range the formal limit, the modules
$E_V,\Lambda_V$ can be defined by these solutions, which thus become the primary objects. 
We may prefer, for convenience, the (tensor) inverse modules and their holomorphic 
solutions; thus, $E_V^{-1}$ is the quotient  of $\cA_3$ by the annihilator 
$\cI_V$ of the (holomorphic) solution $\Gamma_V^{-1}$: 
\[
\begin{array}{ccccc}
  E_V^{-1}:=&\cA_3/\cI_V &\xrightarrow{\ \cdot\: \Gamma_V^{-1}\ }& \cO_1 = \cA_3/\cI_1\\  
  & \cup & & \\
  &  \cO_1 & &
\end{array}\]
If we regard the quotient $\cA_3/\cI_V$ as an analytic sheaf over the Toda base, the solution 
map $\Gamma_V^{-1}$ is an isomorphism. (Otherwise, its infinitely many zeroes prevent 
it from surjecting onto $\cO_1$.) We can then characterise $\cN\cC^\circ(G;E)$ in three 
equivalent ways, the last two of which are $\Gamma$-conjugate:  
\begin{enumerate}\itemsep0ex
\item as the subring of elements of $\cA_3$ which survive conjugation by $\Gamma_V^{-1}$, 
\item as the subring of $\cA_3$ whose multiplicative action preserves the inclusion $\cO_1\subset E_V^{-1}$, 
\item as the subring of $\cA_3$ whose multiplicative action preserves the inclusion 
$\Gamma^{-1}_V\cO_1 \subset \cO_1$. 
\end{enumerate}
There is a parallel story for $\cN\cC_4$. 
Before spelling out the details, let us revisit the case of $\mathrm{U}_1$.

\subsection{Example I: $\mathrm{U}_1$ with its standard representation.}\label{noncomu1}
The symplectic space $T^\vee\bC^\times = \mathrm{Spec}\,H_*^{\mathrm{U}_1}(\Omega\mathrm{U}_1)$ 
has a natural non-commutative deformation, realised topologically by the Pontryagin ring  
$H_*^{\mathrm{U}_1\times R}(\Omega\mathrm{U}_1)$. Indeed, on $\pi_1\mathrm{U}_1$, $z$-multiplication 
is the shift $n\mapsto n+1$, at which point the $R$-rotation collects an extra $\mathrm{U}_1$-weight. 
We compute from here the Pontryagin ring as $\bC[h]\langle z^\pm,\tau\rangle$ with relation $z\tau=(\tau+h)z$. 
We now identify the non-commutative Coulomb branch $H_*^{\mathrm{U}_1\times R}(L_L)$ for 
the standard representation $L$:

\begin{lemma}\label{CXY}
$\cN\cC_3(\mathrm{U}_1;L\oplus L^\vee)$ is the subring of  
$H_*^{\mathrm{U}_1\times R}(\Omega\mathrm{U}_1)$ generated over $\bC[h]$ by $z, z^{-1}\tau$.
\end{lemma}

\begin{remark} Setting $X=z, Y=z^{-1}\tau$, this ring is $\bC[h]\langle X,Y\rangle/ ([X,Y]-h)$, 
as one could  have guessed from the Poisson relation $\{x,y\}=h$ in $\bC[x,y]$ 
(notation as in Example~\ref{keyx}). 
\end{remark}
\begin{proof}
Using the right inclusion in \S\ref{twoemb} to embed the ring, we find at the winding mode 
$n\ge 0$ the summand $z^n\cdot\bC[h,\tau]$; whereas at a negative 
winding mode $(-n)$, we find
\[
z^{-n}\mathrm{e}_- = z^{-n}\tau(\tau+h)\cdot\dots\cdot(\tau+(n-1)h) = (z^{-1}\tau)^n,
\]  
from the Euler class $\mathrm{e}_-(-n)$ of $I_-$, 
which is the summand missing from the right copy of $V[\![z]\!]$.
\end{proof}

Recall now the $h$-periodic Gamma-function 
\[
\Gamma(w;h):= h^{w/h-1}\Gamma(w/h). 
\]
It satisfies 
\[
\Gamma(w+h;h)=w\Gamma(w;h) \text{ and }\Gamma(h;h)=1. 
\]
From $z\tau z^{-1} =\tau+h$ we get
\begin{equation}\label{conj}
\Gamma(\tau;h)\cdot z\cdot \Gamma(\tau;h)^{-1} = \tau^{-1} z, 
\end{equation}
which exhibits $\Gamma(\tau;h)$ as a solution to the module $\cA_3/(z-\tau)$, 
the obvious quantisation of $\vep_V$: 

\begin{corollary}\label{solution}
Away from the poles, sending $1$ to $\Gamma(\tau;h)$ maps $\cA_3/(z-\tau)$ into 
$\cO_1 =\cA_3/(z-1)$. \qed
\end{corollary}

\noindent Holomorphy of the reciprocal function $\Gamma^{-1}$ 
is a reason to prefer the inverse module $\cA_3/(1-\tau z)$. 

\begin{remark}\label{stirling}
As $h\to 0$, Stirling's approximation gives (when $|\arg(\tau/h)|< \pi^-$) 
\[
\log\Gamma(\tau;h) = \frac{\tau}{h}(\log\tau -1) -\frac{1}{2}\log h + \frac{1}{2}\log(2\pi/\tau) 
+ O(h/\tau), 
\]
and we find in the leading $h^{-1}$ coefficient the Legendre transform $\psi^*(\tau)$ of $\psi(z)=z$. 
The Legendre correspondence quantises to the Laplace transform: viewing $\cA_3$ as the ring of 
$\mathscr{D}_h$-modules on $\bC^\times$, with $\tau=h\cdot z\frac{\partial}{\partial z}$, we  
find that the function $\exp(-z/h)$ on $\bC^\times$ is the solution to the module $\mathscr{D}_h/(\tau+z)$, 
Laplace transformed from the one in Corollary~\ref{solution}. 
\end{remark}
\begin{proposition}\label{noncom}
$\cN\cC_3(\mathrm{U}_1;L\oplus L^\vee)$ is the subring of elements of 
$H_*^{\mathrm{U}_1\times R}(\Omega\mathrm{U}_1)$ which survive 
conjugation by $\Gamma(\tau;h)^{-1}$. \qed
\end{proposition}

\begin{proof}
Survival of $z$ and $z^{-1}\tau$ is clear from \eqref{conj}. To show 
the converse inclusion, choose an $\cA_3$-element of negative $z$-degree $(-n)$. 
Reordering factors expresses it uniquely in monomials of the form 
\[
(z^{-1}\tau)^n\tau^m, \quad m\ge0, \quad \text{and}\quad (z^{-1}\tau)^az^{a-n},\quad 0\le a<n. 
\]
The former survive $\Gamma^{-1}$-conjugation. To rule out the latter, note that conjugation 
converts them to $z^{-a}(\tau z)^{a-n}$. These monomials are not regular 
in any $\bC[h]$-linear combination, or else a right multiplication 
by $(\tau z)^n$ would lead to a linear dependence among the monomials
\[\begin{split}
z^{-a}(\tau z)^a &= (\tau-h)\cdot\dots\cdot(\tau-ah), \qquad\qquad 0\le a <n \\  
z^{-n}\tau^m(\tau z)^n &= (\tau-nh)^m\cdot(\tau-h)\dots\cdot(\tau-nh), \quad m\ge 0
\end{split}
\]
which is pre-empted by their $\tau$-degree.
\end{proof}

\subsection{The $\Gamma_V$-class.}\label{gamma}
Generalising this involves promoting  $\Gamma$ to a 
multiplicative characteristic class of complex vector bundles. This requires some care: 
the Hirzebruch construction, the product $\Gamma(F;h) := \prod_\rho\Gamma(\rho;h)$ over the 
Chern roots $\rho$ of $F$, is ill-defined, as $\Gamma$ has a pole at $0$. The reciprocal 
$1/\Gamma$ is entire holomorphic, but its vanishing at $0$ would lead to an unstable class, 
undefined for virtual bundles. One remedy is to 
include the equivariant scaling (mass) parameter $\mu$, resulting in a $\mu$-meromorphic 
calculus for the classes $\prod_\rho\Gamma(\mu+\rho;h)$. Thus, a representation $V$ of $G$ 
leads to the entire holomorphic (in $\mu, \xi$) reciprocal function 
\[
\Gamma_V(\xi,\mu;h)^{-1}:\frh_\bC/W\times\bC \to \bC, \quad (\xi,\mu)\mapsto 
	\det\nolimits_V \Gamma(\xi\oplus\mu;h)^{-1}. 
\] 

\begin{remark}[Massless specialisation.] The correct massless specialisation is 
$\mu = \frac{1}{2}h$ (not $\mu=0$). In the construction of \cite{bfn},
this should be interpreted as inserting a square root of the canonical bundle on the doubled 
disk $\ddisk$. The same insertion within the index bundle does away with the vanishing 
condition at $\infty$ in the constructions of \S\ref{proofs}. The specialisation is illustrated by 
the identity $\Gamma\left(\frac{h}{2}+\tau;h\right)\Gamma\left(\frac{h}{2}-\tau;h\right) = \frac{\pi}{h}
\sec\left(\frac{\pi\tau}{h}\right)$: the product is therefore anti-central in $\cN\cC(\mathrm{U}_1)$ 
(it conjugates $z$ to $(-z)$), generalizing the identity 
$\mathrm{e}_L\cup\mathrm{e}_{L^\vee} = (-1)^n$ of Example~II in \S2.
\end{remark}

\begin{remark} \label{gammaeuler}
Interpreting $h$ as the equivariant parameter of the loop rotation group $R$, the Weierstra{\ss} 
product expansion portrays $\Gamma_V^{-1}$ as a regularised Euler class of the space of Taylor 
loops $V[\![z]\!]$. This interpretation also makes sense over 
certain stacks with a circle action, such as $G[\![z]\!]\backslash\Omega^aG$: 
a reasonable demand is that their $R$-equivariant homology is free over $\bC[h]$, so 
that extension of scalars to functions of $h$ holomorphic off the negative real axis 
(and allowing poles in $\mu,\xi$ as needed) is a faithful operation. For a torus, we can 
always pretend that $h$ is a numerical parameter, because the $R$-action on 
the stack $\mathfrak{Pic}(\bP^1)$ is trivializable. 
\end{remark}

\subsection{Example II: $\mathrm{U}_1$ with a representation $V$.} Split  
$V=V_+\oplus V_-$ according to $z$-exponents\footnote{A trivial representation 
summand $V_0$ does not affect the Coulomb branch.}, writing $\Gamma_V=\Gamma_+\Gamma_-$. 
We have
\begin{equation}\label{plusminusconj}
\Gamma_V^{-1}\cdot z\cdot \Gamma_V = (\Gamma_+\Gamma_-)^{-1}\cdot z\cdot\Gamma_+\Gamma_-
	= \Gamma_+^{-1}z\Gamma_+z^{-1} \cdot z \cdot z^{-1}\Gamma_-^{-1}z\Gamma_- = 
	\mathrm{e}_+(1)\cdot z\cdot \mathrm{e}_-(1)^{-1}
\end{equation}
with $\mathrm{e}_\pm(1)$ the $R$-equivariant extensions of the index Euler classes 
of \S\ref{indexsplit} at $\eta=1$. Repeating the computation in the proof of Lemma~\ref{CXY},  
\begin{equation}\label{eulerpowers}
[z\mathrm{e_-(1)}]^n = z^n\mathrm{e}_-(n), \quad [\mathrm{e}_+(1)z]^n = \mathrm{e}_+(n)z^n, 
\quad n\ge0. 
\end{equation}
\begin{proposition}
$\cN\cC_3^\circ(\mathrm{U}_1;V\oplus V^\vee)$ is generated over $\bC[\mu,\tau,h]$ 
by $z\mathrm{e}_-(1), z^{-1}\mathrm{e}_+(1)$, and is the subring of 
$H_*^{\mathrm{U}_1\times S^1\times R}(\Omega\mathrm{U}_1)$ surviving conjugation 
by $\Gamma_V$.
\end{proposition}

\begin{proof}
From \eqref{plusminusconj} we see that the listed generators survive, and  
\eqref{eulerpowers} shows that their $n$th power generates the summand of degree $\pm n$ 
over the Toda base. 
Fix now $n>0$ say. The need for the $z^n\mathrm{e}_-(n)$ factor in a surviving element 
follows from unique factorisation in the ground ring $\bC[\mu,\tau,h]$. Namely, 
$z^n f(\mu,\tau, h)$ conjugates to $\mathrm{e}_+(n)z^n\mathrm{e}_-(n)^{-1}f$. 
The linear factors of $z^{-n}\mathrm{e}_+z^n$ have the form $(\mu+k\tau -ph)$ with 
$p>0$, and are prime to the denominator $\mathrm{e}_-(n)$, whose factors carry 
non-negative multiples of $h$ going with $\mu$; so all canceling factors  must come from $f$. 
\end{proof}

Localising on the Toda base, we find from the Abelian calculation, formally close to
$h=0$: 
\begin{corollary}
$\Gamma_V$ conjugates the unit module of $\cN\cC^\circ_3$ into a module $E_V$ with 
support $\bar{\vep}_V$.\qed
\end{corollary}
\noindent
Away from formal $h=0$, we can define the convolution-inverse module $E_V^{-1}$  as the quotient 
of $\cA_3$ by the annihilator of $\Gamma_V^{-1}$. Sending $1\in\cA_3$ to 
$\Gamma_V^{-1}$ identifies it with $\cO_1$.

\subsection{Description of the $\cN\cC$ spaces.}
Theorems~4 and~5 below are quantum versions of the Lagrangian-shift description of the 
Coulomb branches. The proof follows the commutative argument, with its core relying on 
the Euler interpretation of $\Gamma_V$ (Remark~\ref{gammaeuler}): conjugation by 
$\Gamma_V^{-1}$ becomes capping with (the $R$-equivariant) $\mathrm{e}_V$. 
The capping operation is of course canonical, but the left and right module structures of $\cA_3$ 
over $H^*_{G\times R}$ differ, as they come from left and right pull-backs from $B(G\times R)$ 
to the stack $R\ltimes(G\ltimes LG\rtimes G)$. Of course, this is why $\Gamma_V$-conjugation 
is not trivial. 

\begin{maintheorem}\label{cnc}
The non-commutative deformation  $\cN\cC_3^\circ(G;E)$, defined as 
$H_*^{G\times S^1\times R}(L_V)$, comprises those elements of $H_*^{G\times S^1\times R}(\Omega G)$ 
which survive conjugation by $\Gamma_V^{-1}$. 
\end{maintheorem}

\begin{proof} Incorporate the $R$-action in the embeddings $\varphi_{l,r}$ of \S\ref{twoemb}. 
I claim that conjugation switches $\varphi_r$ to $\varphi_l$: this need only be checked 
generically on the Coulomb branch, and can be seen by restriction to the maximal torus, 
reducing to the abelian calculation in Example~II above. 

It follows that $\varphi_r$ places $\cN\cC_3(G;E)$ within the surviving subring, and 
the argument closes by quoting Proposition~\ref{noncom} on each Schubert stratum $C_\eta$, 
with $z^\eta$ in lieu of $z$,  to conclude that $\mathrm{Gr}_\eta\cN\cC_3^\circ(G;E)$ 
exhausts the surviving part of $\mathrm{Gr}_\eta\cN\cC_3^\circ$. 
\end{proof}

\subsection{The space $\cN\cC_4^\circ$.} In the Key Example of $\mathrm{U}_1$, the Pontryagin ring 
$K_*^{\mathrm{U}_1\times R}(\Omega\mathrm{U}_1)$ is the standard non-commutative (complexified) torus, 
$\bC[q^\pm]\langle t^\pm,z^\pm\rangle$ with relation $zt=qtz$. To proceed, we need Jackson's 
$p$-Gamma function \cite{jack}. In terms of $p$-Pochhammer symbols $(x;p)_\infty = 
\prod_{n\ge 0}(1-xp^n)$, convergent for $|p|<1$, this is 
\[
\Gamma_p(w;h) = (1-p)^{1-w/h}\frac{(p^h;p^h)_\infty}{(p^w;p^h)_\infty}, \quad\text{satisfying}
\quad \Gamma_p(w+h;h) = \frac{1-p^w}{1-p}\Gamma_p(w;h). 
\]
The requisite version of $\Gamma_p$ arises in the limit $p,h\to 0$, as the expansion variables 
$q:=p^{-h}$, $t:=p^{-w}=p^{-\tau}$ are kept finite.\footnote{The signs keep the series convergent for $|q|>1$, matching $h>0$ in the additive case.}  Set
$\Gamma_0(t) := (q^{-1};q^{-1})_\infty/(t^{-1};q^{-1})_\infty$ and note the conjugation
\[
\Gamma_0(t)\cdot z\cdot\Gamma_0(t)^{-1}= (1-t^{-1})^{-1}z,
\]
with the $K$-theoretic Euler class $(1-t^{-1})$ replacing $\tau$ in \eqref{conj}.

\begin{remark}
In analogy with Remark~\ref{stirling}, the Laplace transform of our solution ${\Gamma}_0$ 
is expressed in terms of the $q$-exponential function $e_q$, namely $e_q\left(\frac{z}{1-q}\right) =\left(z;q\right)_\infty^{-1}$.   
\end{remark}

Define now the multiplicative class $\Gamma_{0;V}$ for vector bundles, valued in localised 
equivariant $K$-theory, as in \S\ref{gamma}; formally, near $q=1$, 
we then have 
\begin{proposition}
$\Gamma_{0;V}$ conjugates the unit module of $\cN\cC_4^\circ$ into a module $\Lambda_V$ with 
support $\lambda_V$. \qed
\end{proposition}

Finally, the argument used for $\cN\cC_3$ applies, after working locally on the Toda base, to give

\begin{maintheorem}\label{cnck}
The non-commutative deformation  $\cN\cC_4^\circ(G;V\oplus V^\vee)$, defined by 
$K_*^{G\times S^1\times R}(L_V)$, comprises those elements of $K_*^{G\times S^1\times R}(\Omega G)$ 
which remain regular after conjugation by ${\Gamma}_{0;V}^{-1}$. \qed
\end{maintheorem}

\section*{Appendix: A Weyl character formula for certain Coulomb branches}
Here, I verify the abelianisation result mentioned in the introduction, which describes `most' 
Coulomb branches for $G$ in terms of those for the Cartan subgroup, with their Weyl group symmetry. There is also a non-commutative version, as in \S\ref{noncomgen}; I will return to it in a future paper. 

\begin{maintheorem} For any representation $V$ of $G$ whose weights contain the roots of $\frg$, we have
\[
\cC_{3,4}(G;E)\cong \cC_{3,4}\left(H;E\ominus (\frg/\frh)^{\oplus 2})\right)/W, 
\]
compatibly with the embeddings of \S\ref{twoemb} and the morphism $\cC_{3,4}(G;\mathbf{0}) \to 
\cC_{3,4}(H;\mathbf{0})/W$.  
\end{maintheorem} 
\begin{proof}
Working over the common bases $\frh/W$ and $H/W$, $H$-fixed point localisation shows that the map induced 
by the named morphisms is an isomorphism away from the root hyperplanes; whereas, generically on 
the root hyperplanes, the $\mathrm{SL}_2$ calculation of \S\ref{su2} confirms isomorphy. 
This settles the matter, because the algebras are free $\cO$-modules over the Toda base and agree 
in co-dimension $2$. 
\end{proof}
\stepcounter{section}
\begin{remark}
The calculation for $\cC_3$ of $\mathrm{SL}_2$ was seen to hold more generally, for all 
but a few choices of $E$. This generalizes to all groups, by the argument above: however, the 
formulation of the right-hand side needs more care. Exploiting the local descriptions of 
the $\cC_4$ Toda bases in terms of 
$\cC_3$, one can then push this to an awkward but effective calculation of most $\cC_4$ Coulomb branches. 
It would be truly useful to find the formulation which dispenses with all constraints on $E$: 
this might allow an abelianised calculation of Coulomb branches with non-linear matter.
\end{remark}

\end{document}